\documentclass[reqno,12pt]{amsart}
\usepackage{a4wide}
\usepackage{amsmath}
\usepackage{amssymb}
\usepackage{amsthm}
\usepackage{mathrsfs}
\usepackage[utf8]{inputenc}

\numberwithin{equation}{section}

\newtheorem{theo}{Theorem}

\newtheorem{prop}{Proposition}

\newtheorem{lem}{Lemma}

\theoremstyle{remark}

\def\om{\omega}

\def\({\left(}
\def\){\right)}
\def\[{\left[}
\def\]{\right]}

\def\li{\operatorname{Li}}

\def\dd{\textup{d}}

\newcommand{\Z}{\mathbb{Z}}
\newcommand{\Q}{\mathbb{Q}}

\newcommand{\C}{\mathbb{C}}

\newcommand{\calF}{{\mathcal F}}

\newcommand{\BBB}{{\bf H}}
\newcommand{\opa}{\BBB^{n+1}_{n+1,0}\BBB^{n+1}_{n+1,n+1}}
\newcommand{\opb}{\BBB^{n+1}_{0,n+1}\BBB^{n+1}_{n+1,n+1}}
\newcommand{\gdo}{\mathcal{O}}
\newcommand{\Prn}{\mathcal{P}_{r,n}}
\newcommand{\Qrn}{\mathcal{Q}_{r,n}}
\newcommand{\soro}{\mathcal{R}_{r,n}}

\newcommand{\mzero}{\mathcal{M}_0}

\begin{document}

\title[]{Multiple zeta values, Pad\'e approximation and Vasilyev's conjecture}

\author[]{S. Fischler and T. Rivoal}
\date{\today}



\begin{abstract} Sorokin gave in 1996 a new proof that $\pi$ is transcendental. 
It is based on a simultaneous Pad\'e approximation problem involving certain multiple 
polylogarithms, which evaluated at the point $1$ are multiple zeta values equal to powers of $\pi$. 
In this paper we construct a Pad\'e approximation problem  of the same flavour, and prove 
that it has a unique solution up to proportionality. At the point $1$, this provides a 
rational linear combination of $1$ and multiple zeta values in an extended sense that turn 
out to be values of the Riemann $\zeta$ function at odd integers. As an application, we obtain  a new 
proof of Vasilyev's conjecture  for any odd weight, concerning the explicit evaluation of certain 
hypergeometric multiple integrals; it  was first proved by Zudilin in 2003. 
\end{abstract}

\maketitle

\section{Introduction}
The goal 
of this paper is to provide a completely new proof of Vasilyev's conjecture for any odd weight $d\ge 3$ 
by solving a simultaneous Pad\'e approximation problem involving multiple polylogarithms.  Before explaining in details 
our approach, we provide some background. Vasilyev~\cite{vasiliev} conjectured in 1996, 
that for any integers $d\ge 2$ and $n\ge 0$, 
\begin{equation}\label{eq:intro1}
J_{d,n}:=\int_{[0,1]^d} \frac{\prod_{j=1}^d x_j^n(1-x_j)^n}{Q_d(x_1,  \ldots, x_d)^{n+1}} \dd x_j \in \mathbb Q + 
\mathbb Q \zeta(2+e_d) + \mathbb Q \zeta(4+e_d) + \cdots +  \mathbb Q \zeta(d)
\end{equation}
where $e_d=0$ if $d$ is even, $e_d=1$ otherwise, and $Q_1(x_1):=1-x_1$,
\begin{align*}
Q_d(x_1, \ldots, x_d):&=1-Q_{d-1}(x_1, \ldots, x_{d-1})x_d, \quad d\ge 2
\\                   &=1-(1-(\cdots 1-(1-x_1)x_2\cdots )x_{d-1})x_d.
\end{align*}
This conjecture was already known to be true for $d=2$ and $d=3$, since Beukers~\cite{beukers} used these integrals to 
get new and quick versions of Ap\'ery's proofs \cite{Apery} of the irrationality of $\zeta(2)$ and $\zeta(3)$. 
Vasilyev himself proved his conjecture in 
the cases $d=4$ and $d=5$, results which in fact led him to the conjecture. The first complete proof 
was given by Zudilin~\cite{zudilin} who showed that $J_{d,n}$ is   equal to a very-well-poised hypergeometric series 
whose value was already known to be  in $\mathbb Q + \mathbb Q \zeta(2+e_d) + \mathbb Q \zeta(4+e_d) + \cdots 
+  \mathbb Q \zeta(d)$. Two other proofs of Vasilyev's conjecture  were subsequently found, 
one by Zlobin~\cite{zlobin} (direct attack) and another indirect one by 
Krattenthaler-Rivoal~\cite{kratriv} (limiting case of Andrews' hypergeometric identity, in the spirit of Zudilin). The 
fourth one, given in the present paper, is completely different since it relies on solving a 
simultaneous Pad\'e approximation problem involving multiple polylogarithms.  

\bigskip

To state this problem we need some notations. 
Given any finite word $\sigma$ built on a (possibly infinite) alphabet $\{a,b, \ldots\}$, 
we denote by   $\{\sigma\}_j:=\sigma\sigma\cdots \sigma$   the 
concatenation $j$ times of $\sigma$. By convention, $\{\sigma\}_0=\emptyset$. We will use two alphabets, 
namely $\mathbb N^{*} = \{1,2,\ldots\}$ 
and $\{\ell,s\}$. We consider  multiple polylogarithms in the following extended sense:
\begin{equation}\label{eq:intro4}
\li_{b_1b_2\cdots b_p}^{a_1a_2\cdots a_{p-1}}(z):=
\sum_{k_1\gtrsim k_2\gtrsim \cdots \gtrsim k_p\ge 1} \frac{z^{k_1}}{k_1^{b_1}k_2^{b_2}\cdots k_p^{b_p}}
\end{equation}
where $\vert z\vert <1$, $b_j\in \mathbb N^{*} $ and $a_j\in\{\ell,s\}$ for all $j$. For $j=1, \ldots, p-1$, 
the symbol $\gtrsim\;\in \{>, \ge \}$ in $k_j\gtrsim k_{j+1}$ 
is determined by the following rule: it is set to $>$ if $a_j=s$,  and   
to $\ge$ if $a_j=\ell$. In this way, $s$ stands for a {\em strict} inequality, 
and $\ell$ for a {\em large} one. If $a_j = s$ for any $j$ we obtain the usual multiple 
polylogarithm $\li_{b_1b_2\cdots b_p} (z)$; if $a_j = \ell$ for any $j$ we obtain the variant 
denoted by $\operatorname{La}_{b_1b_2\cdots b_p} (z)$ in \cite{crefiri} and 
by $\operatorname{Le}_{b_1b_2\cdots b_p} (z)$ by  
Ulanski\u{\i} and Zlobin. Sorokin used in~\cite{sorokin1} the functions 
 $\li_{\{1\}_{2r+1}}^{\{s\ell\}_{r}}(1-x)$ and $\li_{\{1\}_{2r}}^{\{ \ell s\}_{r-1} \ell}(1-x)$, 
which he denoted respectively by $\varepsilon_r(x)$ and $\varphi_r(x)$. In this paper, all multiple polylogarithms $
\li_{b_1b_2\cdots b_p}^{a_1a_2\cdots a_{p-1}}(z)$  will be considered for $z\in\C\setminus [1,\infty)$ using analytic continuation.
As usual, the integer $p$ in \eqref{eq:intro4}   is called the {\em depth}, and $b_1+\cdots+b_p$ is the {\em weight}.

\bigskip

Our main result is the explicit resolution of the following simultaneous Pad\'e approximation problem. 
Given   integers $n,r\ge 0$, we want to find 
polynomials $A_{\rho,r,n}(z)$,  $B_{\rho,r,n}(z)$, $C_{\rho,r,n}(z)$, 
$D_{r,n}(z) \in \mathbb C[z]$, for $0\leq \rho \leq r$, all of degree at most $n$, such that
\begin{align*}
S_{r,n}(z):=&\sum_{\rho=0}^r \bigg[A_{\rho,r,n}(z) \li_{2\{1\}_{2\rho+1}}^{\{\ell s\}_\rho\ell}\Big(\frac1z\Big)
+B_{\rho,r,n}(z) \li_{\{1\}_{2\rho+2}}^{\{\ell s\}_\rho\ell}\Big(\frac1z\Big)
\\
&\hspace{2cm} +C_{\rho,r,n}(z) \li_{\{1\}_{2\rho+1}}^{\{s\ell\}_\rho}\Big(\frac1z\Big)\bigg]+D_{r,n}(z)
= \mathcal{O}\Big(\frac1{z^{(r+1)(n+1)}}\Big)
\\
U_{j,r,n}(z):=&\sum_{\rho=j}^r A_{\rho,r,n}(z) \li_{1\{2\}_{r-\rho}}^{\{\ell\}_{r-\rho}}(1-z) 
+ B_{j,r,n}(z) = \mathcal{O}\big((1-z)^{n+1}\big), 
\quad j=0, \ldots, r
\\
V_{j,r,n}(z):=&\sum_{\rho=j}^r A_{\rho,r,n}(z) \li_{\{2\}_{r-\rho+1}}^{\{\ell\}_{r-\rho}}(1-z) 
+ C_{j,r,n}(z) = \mathcal{O}\big((1-z)^{n+1}\big), \quad j=0, \ldots, r.
\end{align*} 
We will denote by $\mathcal{P}_{r,n}$ this Pad\'e approximation  problem. The various symbols 
$\mathcal{O}$ have the following 
meaning. The function $S_{r,n}(z)$ is obviously analytic at $z=\infty$ and we ask   its order there to  be at 
least $(r+1)(n+1)$. Similarly, the functions $U_{j,r,n}(z)$ and $V_{j,r,n}(z)$ are analytic at $z=1$ and we ask 
  their orders there to be at least $n+1$. This is a mixed  Pad\'e approximation problem, 
namely in between type $I$ problems and 
type $II$ problems. 
  Similar mixed Pad\'e approximation problems often occur in the Diophantine theory of (multiple) 
zeta values; see for instance~\cite{firi, sorokin2, sorokin1}.

\medskip

The problem $\mathcal{P}_{r,n}$ can be trivially converted into a linear algebra problem: it amounts to solving a 
system of $(3r+4)(n+1)-1$ linear equations in  $(3r+4)(n+1)$ unknowns (the coefficients 
of the polynomials). Hence, there is at least one non identically zero solution. 
Our main theorem shows that the solution is unique up to a multiplicative constant.

\medskip

\begin{theo}\label{theo:1} For any integers $n,r\ge 0$, 
the function $S_{r,n}(z)$ in $\mathcal{P}_{r,n}$ is given by the following  
hypergeometric integral (up to a multiplicative constant), which converges for any 
$z\in \mathbb C\setminus[0,1)$:
\begin{multline}\label{eq:intro6}
S_{r,n}(z)=(-1)^{n+1}z^{(r+1)(n+1)}
\\
\times \int_{[0,1]^{2r+3}} 
\frac{\displaystyle u_0^{(r+1)(n+1)-1}(1-u_0)^n 
\prod_{j=1}^{r+1}\big((u_jv_j)^{(r-j+2)(n+1)-1}(1-u_j)^n(1-v_j)^n\big)}
{\displaystyle \prod_{j=1}^{r+1}\big((z-u_0u_1v_1\cdots u_{j-1}v_{j-1}u_j)^{n+1}
(z-u_0u_1v_1\cdots u_{j}v_{j})^{n+1}\big)} \dd{\bf u}\dd{\bf v}.
\end{multline}
\end{theo}

For $r=0$, the problem $\mathcal{P}_{0,n}$ and the integral for  
$S_{0,n}(z)$  exactly match those considered by Sorokin in~\cite{sorokin2}, from which he deduced 
a new proof of Ap\'ery's theorem. However, our derivation of the integral for 
$S_{0,n}(z)$ is different from Sorokin's.

\bigskip

For any $r\geq 0$, the integral representation \eqref{eq:intro6}  provides a new proof of 
Vasilyev's conjecture, by taking $z=1$ (see \S~ \ref{sec:Vasilyev} for details). It 
would be very interesting 
to obtain a new proof of the infiniteness of irrational values among the $\zeta(2r+1)$ (see~\cite{BR, rivoal}) 
by  solving a Sorokin-type Pad\'e problem involving multiple polylogarithms as in Theorem \ref{theo:1}, as Sorokin 
did \cite{sorokin2} for Ap\'ery's 
theorem  (see \S~ \ref{subsecinf} at the end of the paper).

\medskip

 Theorem \ref{theo:1} is based on Sorokin's proof \cite{sorokin1}  of the 
transcendence of $\pi$, which relies  on the resolution of a simultaneous Pad\'e approximation 
problem involving certain multiple 
polylogarithms (see \S~\ref{sec:sorokinpi2} for details), as well as on the identity 
$\li_{\{2\}_{r}}^{\{s\}_{r-1}}(1)=\frac{\pi^{2r}}{(2r+1)!}$ for any integer $r\ge 1$.

\medskip

The integral for $S_{r,n}(z)$ can be used to get explicit expression of the polynomials, all of which obviously 
  have rational coefficients. This can be done by various means, for instance one can convert 
the integral into the series
\begin{multline*}
S_{r,n}(z)=
\\n!\!\!\sum_{k_0\ge \cdots \ge k_{2r+1}\ge 1} \frac{(k_0-k_1+1)_n(k_1-k_2+1)_n\cdots 
(k_{2r}-k_{2r+1}+1)_n(k_{2r+1}-n)_n}{\displaystyle \prod_{j=0}^r 
\big((k_{2j}+(r-j+1)(n+1))_{n+1}^{e_j}(k_{2j+1}+(r-j)(n+1))_{n+1}\big)} \frac1{z^{k_0+r(n+1)}}
\end{multline*}
(where $e_0=2$ and $e_j=1$ for $j\ge 1$) and then use the algorithm described in~\cite{crefiri}.

\bigskip

The paper is organised as follows.   In \S~\ref{sec:Vasilyev}, we deduce Vasilyev's  conjecture for odd values of~$d$ 
from Theorem~\ref{theo:1}. In \S~\ref{sec:technicalities}, we present a few tools needed for the proof of 
Theorem~\ref{theo:1}, in particular an iterative construction of hypergeometric multiple  integrals. 
In \S~\ref{sec:weight}, 
we prove an important representation formula for multiple polylogarithms and derive  a few consequences useful in 
the resolution of $\mathcal{P}_{r,n}$. Section~\ref{sec:proofthm1}, devoted to the proof of Theorem~\ref{theo:1}, 
is decomposed in many steps. The first 
two steps show  how to reduce the problem $\mathcal{P}_{r,n}$ to Sorokin's problem for $\pi^2$ (recalled in 
\S~\ref{sec:sorokinpi2}) and the subsequent steps complete the proof. At last we 
construct in \S~\ref{subsecinf} a family of integrals, containing \eqref{eq:intro6}, 
which enable one to prove that infinitely many odd zeta values  $\zeta(2r+1)$ are irrational \cite{BR, rivoal}.

\section{New proof of Vasilyev's  conjecture for  odd weights} \label{sec:Vasilyev}

To deduce Vasilyev's  conjecture from Theorem~\ref{theo:1}, we first define (when $b_1\geq 2$)   
extended multiple zeta values by
\begin{equation}\label{eq15bis}
\zeta_{b_1b_2\cdots b_p}^{a_1a_2\cdots a_{p-1}}:=\li_{b_1b_2\cdots b_p}^{a_1a_2\cdots a_{p-1}}(1) = 
\sum_{k_1\gtrsim k_2\gtrsim \cdots \gtrsim k_p\ge 1} \frac{1}{k_1^{b_1}k_2^{b_2}\cdots k_p^{b_p}}
\end{equation}
with the same definition for the symbols $\gtrsim$ as in Eq. \eqref{eq:intro4}.  
In particular, when $a_j=s$ for all $j$, we have the usual multiple zeta values 
$\zeta_{b_1b_2\cdots b_p}^{ \{s\}_{p-1}}=\zeta(b_1,b_2,\ldots, b_p)$.

\bigskip

Then we   remark that the  Pad\'e conditions for the functions $U_{j,r,n}(z)$ and $V_{j,r,n}(z)$ in $\Prn$ 
ensure that all  polynomials 
$B_{j,r,n}(z)$ and $C_{j,r,n}(z)$ vanish at $z=1$ ($j=0, \ldots, r$). Since multiple polylogarithms have (at most) a 
logarithmic singularity at $z=1$, this implies that when we take the limit $z\to 1$ in~\eqref{eq:intro6}, we get 
\begin{align*}
(-1)^{n+1} \int_{[0,1]^{2r+3}} &
\frac{\displaystyle u_0^{(r+1)(n+1)-1}(1-u_0)^n \prod_{j=1}^{r+1}\big((u_jv_j)^{(r-j+2)(n+1)-1}(1-u_j)^n(1-v_j)^n\big)}
{\displaystyle \prod_{j=1}^{r+1}\big((1-u_0u_1v_1\cdots u_{j-1}v_{j-1}u_j)^{n+1}(1-u_0u_1v_1\cdots u_{j}v_{j})^{n+1}\big)} 
\dd{\bf u}\dd{\bf v} 
\\
&=\sum_{\rho=0}^r A_{\rho,r,n}(1) \zeta_{2\{1\}_{2\rho+1}}^{\{\ell s\}_\rho\ell} 
+D_{r,n}(1)
\end{align*}
where $A_{\rho,r,n}(1)$ and $D_{r,n}(1)$ are rational numbers. 
Moreover, it is proved in~\cite[Corollaire~8]{fischler2} that  this multiple integral  is equal 
to $J_{2r+3,n}$ for any integer $r\ge 0$ (see also \S~ \ref{subsecinf} below). 
To complete the proof of Vasilyev's   conjecture in this case, we simply need the following 
result, which plays the same role for us as the identity $\li_{\{2\}_{r}}^{\{s\}_{r-1}}(1)=\frac{\pi^{2r}}{(2r+1)!}$ for 
Sorokin in~\cite{sorokin1}.

\begin{prop}\label{prop:3} For any integer $k\ge 1$, we have
\begin{equation}\label{eq:intro7}
\zeta_{2\{1\}_{2k-1}}^{\{\ell s\}_{k-1}\ell}=\zeta_{\{2\}_k 1}^{\{\ell\}_{k}} = 2\zeta(2k+1).
\end{equation}
\end{prop}
\begin{proof}
The second equality in~\eqref{eq:intro7} is due to Zlobin~\cite{zlobin2}.
To prove the first equality, which we haven't found in the literature,
we use the 
representation of (extended) multiple zeta values as Chen iterated integrals. Indeed, 
we have
\begin{align*}
&\zeta_{2\{1\}_{2k-1}}^{\{\ell s\}_{k-1}\ell} 
\\ &\quad=
 \int\limits_{\{0\le x_{2k+1}\le \cdots \le x_1\le 1\}} 
\frac{\dd {\bf x}}{x_1x_2(1-x_2)(1-x_3)x_4(1-x_4)(1-x_5)\cdots x_{2k}(1-x_{2k})(1-x_{2k+1})}
\\
&\quad= \int\limits_{\{0\le y_{2k+1}\le \cdots \le y_1\le 1\}} 
\frac{\dd {\bf y}}{y_1y_2(1-y_2)y_3y_4(1-y_4)y_5\cdots y_{2k}(1-y_{2k})(1-y_{2k+1})}
=\zeta_{\{2\}_k 1}^{\{\ell\}_{k}},
\end{align*}
where we have made the change of variables $x_j=1-y_{2k+2-j}$, $j=1, \ldots, 2k+1$.
\end{proof}

\section{General results on multiple  polylogarithms}\label{sec:technicalities}

We gather in this section various results, useful in the proof of Theorem \ref{theo:1} 
but which may also be of independent interest.

\subsection{Differentiation rules for multiple polylogarithms}\label{sec:diff}

In this section, we describe how to differentiate a multiple polylogarithm. To begin with, 
we state formulas of which the proofs 
are straightforward; we will use them without further mentions.
The letter ${\bf a}$ denotes a finite 
word built on the alphabet $\{\ell,s\}$, the letter ${\bf b}$ a   finite word built on the 
alphabet $\mathbb{N}^{*}$,  and $t$ any integer $\ge 2$.

\begin{align*}
\frac{\dd }{\dd z}& \li_1(z)= \frac{1}{1-z}, \hspace{3.4cm} \frac{\dd }{\dd z}
\Big[ \li_1\Big(\frac1z\Big) \Big]= \frac{1}{z(1-z)},
\\
\frac{\dd }{\dd z}& \li_{1{\bf b}}^{\ell{\bf a}}(z)=\frac{1}{z(1-z)} 
\li_{{\bf b}}^{{\bf a}}(z), \hspace{1.5cm}
\frac{\dd }{\dd z}\Big[  \li_{1{\bf b}}^{\ell{\bf a}}\Big(\frac1z\Big) \Big]=\frac{1}{1-z} 
\li_{{\bf b}}^{{\bf a}}\Big(\frac1z\Big),
\\
\frac{\dd }{\dd z}& \li_{t{\bf b}}^{\ell{\bf a}}(z)=\frac{1}{z} 
\li_{(t-1){\bf b}}^{\ell{\bf a}}(z), \hspace{2.1cm} 
\frac{\dd }{\dd z}\Big[  \li_{t{\bf b}}^{\ell{\bf a}}\Big(\frac1z\Big) \Big]=-
\frac{1}{z} \li_{(t-1){\bf b}}^{\ell{\bf a}}\Big(\frac1z\Big),
\\
\frac{\dd }{\dd z}& \li_{1{\bf b}}^{s{\bf a}}(z)=\frac{1}{1-z} 
\li_{{\bf b}}^{{\bf a}}(z), \hspace{2.1cm} 
\frac{\dd }{\dd z}\Big[ \li_{1{\bf b}}^{s{\bf a}}\Big(\frac1z\Big) \Big]=\frac{1}{z(1-z)} 
\li_{{\bf b}}^{{\bf a}}\Big(\frac1z\Big),
\\
\frac{\dd }{\dd z}& \li_{t{\bf b}}^{s{\bf a}}(z)=\frac{1}{z} 
\li_{(t-1){\bf b}}^{s{\bf a}}(z), \hspace{2.1cm} 
\frac{\dd }{\dd z}\Big[  \li_{t{\bf b}}^{s{\bf a}}\Big(\frac1z\Big) \Big]=-
\frac{1}{z} \li_{(t-1){\bf b}}^{s{\bf a}}\Big(\frac1z\Big).
\end{align*}

We now state a general lemma, whose proof can be done by induction using the formulas above.

\begin{lem}  \label{lem:4} Let $d,n\geq 0$, and $A(z)\in \mathbb C[z]$ be a polynomial of degree $\le d$. Then 
we have
$$
\frac{\dd^{n+1}}{\dd z^{n+1}} \big(A(z) \li_{ b_1b_2\cdots b_p}^{  a_1a_2\cdots a_{p-1}}(z)\big) 
= \sum_{i=0}^{p+1} \sum_{b'=1}^{b_i}  \frac{\widehat{A}_{i,b'} (z)}{z^{n+1}(1-z)^{n+1}} 
\li_{b' b_{i+1}b_{i+2}\cdots b_p}^{  a_ia_{i+1}\cdots a_{p-1}}(z)
$$ 
for some polynomials $\widehat{A}_{i,b'} (z)$ of degree $\leq d+n+1$; here we let $b_{p+1}=1$ so that 
in the sum there is one term corresponding to $i=p+1$, and the associated polylogarithm is equal to~1. 
\end{lem}

It is not difficult to see that in this lemma,  each polynomial  $\widehat{A}_{i,b'} (z)$ depends 
only on $b_1$, \ldots, $b_{i-1}$, $a_1$, \ldots, $a_{i-1}$, and $b_i-b'$. However we won't use this 
remark in the present paper.

\medskip

Using the above relations in the same way, an analogous lemma yields polynomials $\widehat{A}'_{i,b'} (z)$ 
of degree $\leq d+n+1$ such that 
$$
\frac{\dd^{n+1}}{\dd z^{n+1}} \big(A(z) \li_{ b_1b_2\cdots b_p}^{  a_1a_2\cdots a_{p-1}}(1/z)\big) 
= \sum_{i=0}^{p+1} \sum_{b'=1}^{b_i}  \frac{\widehat{A}'_{i,b'} (z)}{z^{n+1}(1-z)^{n+1}} 
\li_{b' b_{i+1}b_{i+2}\cdots b_p}^{  a_ia_{i+1}\cdots a_{p-1}}(1/z).
$$ 

To take advantage of vanishing conditions like the ones on $U_{j,r,n}(z)$ and $V_{j,r,n}(z)$ 
in the Pad\'e problem $\Prn$, the following lemma is very useful.

\begin{lem}  \label{lemannulation} Let $n' \geq 0$, and $g(z)$ be a function holomorphic 
at $z=1$, such that $g(z) = \gdo \big((z-1)^{n+1}\big)$ as $z\to 1$. Then we have
$$
\frac{\dd^{n+1}}{\dd z^{n+1}} \big(g(z) \li_{ b_1b_2\cdots b_p}^{  a_1a_2\cdots a_{p-1}}(z)\big) 
= \sum_{i=0}^{p+1} \sum_{b'=1}^{b_i} h_{i,b'}(z)
\li_{b' b_{i+1}b_{i+2}\cdots b_p}^{  a_ia_{i+1}\cdots a_{p-1}}(z)
$$ 
for some functions $h_{i,b'} (z)$ holomorphic at $z=1$. As in Lemma \ref{lem:4},  we let $b_{p+1}=1$ so that 
in the sum there is one term corresponding to $i=p+1$, and the associated polylogarithm is equal to 1. 
\end{lem}

In other words, no pole appears at $z=1$ if $g$ vanishes to order at least $n+1$ at this 
point (since polylogarithms have at most a logarithmic divergence at 1).

\subsection{An integral operator} \label{subsec32}

Sorokin solved several Pad\'e approximation problems involving multiple polylogarithms 
(see \cite{sorokin2} and \cite{sorokin1}, amongst other papers), which always led to hypergeometric 
multiple integrals. We define now an integral operator intimately related to his approach (and therefore 
also to Theorem \ref{theo:1}). 

Given integers $a,b,n\geq 0$ and a function $F(z)$, we let 
\begin{equation} \label{eq:12}
\BBB_{a,b}^{n+1}(F)(z) = (-1)^{n+1}z^{n+1-a } \int_0^{1} \frac{u^{a+b-n-2}(1-u)^n}{(u-z)^b}F\Big(\frac zu\Big) \dd u .
\end{equation} 
The assumptions on $F$ and the properties of the function $\BBB_{a,b}^{n+1}(F)$ defined in this 
way are detailed in the following lemma.

\begin{lem} \label{lem:7}
Let $F(z)$ be holomorphic on $\C\setminus[0,1]$ and at $z=\infty$; denote by $\omega\geq 0$ 
its order of vanishing at $\infty$. Given $a,b,n\geq 0$, let $\omega' = \omega + a+b-n-1$ and assume that $\omega' \geq 1$. 

Then $\BBB_{a,b}^{n+1}(F)$ is holomorphic on $\C\setminus[0,1]$ and at $z=\infty$; its  order 
of vanishing at $\infty$ is exactly $\omega'$. Moreover,
\begin{itemize}
\item[$(i)$] Letting $R   = \BBB_{a,b}^{n+1}(F)$, we have
\begin{equation} \label{eqlem7}
F(z)= \frac{1}{n!}z^a(1-z)^b R^{(n+1)}(z) .
\end{equation} 
\item[$(ii)$] If $R(z)$ is a function holomorphic on $\C\setminus[0,1]$ and at $z=\infty$ 
such that $R(\infty)=0$ and Eq. \eqref{eqlem7} holds, then $R= \BBB_{a,b}^{n+1}(F)$.
\end{itemize}
\end{lem}

We shall apply this lemma in two cases: either $F(\infty) = 0$ and $a+b \geq n+1$, or $F$ is 
the constant function $F(z)=1$ and $a+b\geq n+2$. In both cases we have $\omega'\geq 1$, so 
that $ \BBB_{a,b}^{n+1}(F)$ is  holomorphic on $\C\setminus[0,1]$ and at $z=\infty$, and $ \BBB_{a,b}^{n+1}(F)(\infty) = 0$. 

\begin{proof}  Let $G(z) = z^\omega F(z)$; then $G(z)$ is holomorphic on $\C\setminus [0,1]$ 
and at $\infty$, with $G(\infty)\neq 0$. By definition of $\omega'$ we have
$$\BBB_{a,b}^{n+1}(F)(z) = (-1)^{n+1}z^{-\omega' } \int_0^{1} \frac{u^{\omega'-1}(1-u)^n}{(\frac{u}{z}-1)^b}
G\Big(\frac zu\Big) \dd u .$$
Since $\omega'\geq 1$ and $u/z \neq 1$ for any $u\in [0,1]$ (since $z\in\C \setminus[0,1]$), 
this formula shows that   $\BBB_{a,b}^{n+1}(F)$ is holomorphic on $\C\setminus[0,1]$ and at $z=\infty$. 
It has order equal to $\omega'$ at $\infty$ because $G(\infty)\neq 0$. 

To prove $(i)$ and $(ii)$, we perform the change of variable $x = z/u$ and deduce
$$\BBB_{a,b}^{n+1}(F)(z) =(-1)^{n+1} \int_z^{\infty} \frac{(x-z)^n}{x^a(1-x)^b}F(x) \dd x .$$
Then assertions $(i)$ and $(ii)$ follow immediately from the following lemma, obtained from the arguments 
given in \cite[p. 60]{sh}.
\end{proof}

\begin{lem} \label{lem:7anc} Let $R, S$ be functions analytic on a neighborhood of $\infty$, with $R(\infty)=0$. Then:
$$
\frac{1}{n!}R^{(n+1)}(z)=S(z) \Longleftrightarrow R(z)=(-1)^{n+1}\int_z^{\infty} (x-z)^{n}S(x) \dd x.
$$
\end{lem}

\bigskip

For Diophantine applications the value $\BBB_{a,b}^{n+1}(F)(1)$ is often the most interesting one; conditions for 
this value to exist are given by the following lemma, whose proof is straightforward.

\begin{lem} \label{lemenun}
Assume that $b\leq n+1$ and $F(z)$ has (at most) a power of logarithm divergence as $z\to 1$, with 
$z\in\C\setminus[0,1]$. Then  $\BBB_{a,b}^{n+1}(F)(z)$ has also  (at most) a power of logarithm 
divergence as $z\to 1$, with $z\in\C\setminus[0,1]$. 

Moreover, if in addition $b\leq n$ then  $\BBB_{a,b}^{n+1}(F)(z)$ has a finite limit as $z\to 1$, 
with $z\in\C\setminus[0,1]$, and this limit is given by taking $z=1$ in the integral representation 
of Eq.~\eqref{eq:12}, which is then convergent.
\end{lem}

In Pad\'e approximation problems with multiple polylogarithms, multiple integrals appear by 
applying successively integral operators $\BBB_{a,b}^{n+1}$ with various parameters. We shall  write 
$\BBB^{n+1}_{a,b} \BBB^{n'+1}_{a',b'} $ for  $\BBB^{n+1}_{a,b}  \circ \BBB^{n'+1}_{a',b'} $, so that
$\BBB^{n+1}_{a,b} \BBB^{n'+1}_{a',b'} (F) = \BBB^{n+1}_{a,b} ( \BBB^{n'+1}_{a',b'} (F) )$. We shall consider 
in \S~\S~ \ref{subsec54} and \ref{subseccalcint} multiple integrals  of the form 
$$
\BBB^{n_1+1}_{a_1,b_1} \BBB^{n_2+1}_{a_2,b_2} \cdots \BBB^{n_p+1}_{a_p,b_p} (\boldsymbol{1}),
$$
where the $a_j, b_j, n_j$ are non-negative integers and $\boldsymbol{1}$ denotes the function equal to 
$1$ on $\C\setminus[0,1]$; such integrals appear in Sorokin's papers (e.g.,  \cite{sorokin2} and \cite{sorokin1}). 
Lemma \ref{lem:7} gives conditions on the parameters that ensure that this 
integral expression is holomorphic on $\C\setminus[0,1]$ and at $z=\infty$, and Lemma \ref{lemenun} 
plays the analogous role for the behaviour at $z=1$.

\bigskip

In the proof of Theorem \ref{theo:1} we shall use the following result which describes the behaviour of 
this integral operator under the change of variable $z\mapsto 1-z$.

\begin{lem}\label{lem:8} For any integers $a_j, b_j, n_j$, $j=1, \ldots, p$ such that 
$\BBB^{n_1+1}_{a_1,b_1} \BBB^{n_2+1}_{a_2,b_2} \cdots \BBB^{n_p+1}_{a_p,b_p}(\boldsymbol{1})$ is 
holomorphic on $\mathbb{C}\setminus [0,1]$ and at $\infty$, we have 
$$
\BBB^{n_1+1}_{a_1,b_1} \BBB^{n_2+1}_{a_2,b_2} \cdots \BBB^{n_p+1}_{a_p,b_p}(\boldsymbol{1})(1-z) = (-1)^{p+n_1+n_2+\cdots+n_p }
\BBB^{n_1+1}_{b_1,a_1} \BBB^{n_2+1}_{b_2,a_2} \cdots \BBB^{n_p+1}_{b_p,a_p}(\boldsymbol{1})(z)
$$
for all $z\in \mathbb{C}\setminus [0,1]$.
\end{lem}
\begin{proof} This is a  consequence of the following fact. Given $f(z)$, we set $f^{\partial}(z):=f(1-z)$. Then 
$$
R(z)=\BBB^{n+1}_{a,b} (S)(z) \Longleftrightarrow R^{\partial}(z)=(-1)^{n+1}\BBB^{n+1}_{b,a} (S^{\partial})(z).
$$
This equivalence results from  Lemma \ref{lem:7}:
$$
S(z)=\frac{1}{n!}z^{a}(1-z)^{b}  R^{(n+1)} (z) 
\Longleftrightarrow
S(1-z)=\frac{(-1)^{n+1}}{n!}z^{b}(1-z)^{a}\big(  R(1-z)\big)^{(n+1)}.
$$
\end{proof}

\subsection{Functional linear independence of polylogarithms}

The extended multiple po\-lylogarithms introduced in  the introduction are very useful to 
state and prove our result, but they are not really {\em new} functions: they are linear 
combinations over $\Z$ of usual multiple polylogarithms (corresponding to $\alpha_1 
= \ldots = \alpha_{p-1} =  s$ in \eqref{eq:intro4}). This follows from the following 
elementary relation (which is the starting point of \cite{crefiri}):
\begin{equation} \label{eqdeplin}
\li_{b_1b_2\cdots b_p}^{a_1 \cdots a_{j-1} \ell a_{j+1} \cdots  a_{p-1}}(z) =
 \li_{b_1b_2\cdots b_p}^{a_1 \cdots a_{j-1} s a_{j+1} \cdots  a_{p-1}}(z)  
+  \li_{b_1 \cdots b_{j-1} b' b_{j+2} \cdots  b_p}^{a_1 \cdots a_{j-1}   a_{j+1} \cdots  a_{p-1}}(z)  
\end{equation} 
where $b' = b_j + b_{j+1}$.

In the proof of Theorem \ref{theo:1} we shall use the following result.

\begin{lem} \label{lemmerom} For any $k$, let ${\bf a}_k$ be a word on the alphabet 
$\{\ell,s\}$ of length $k-1$, with ${\bf a}_1 = {\bf a}_0 = \emptyset$. Then the  
polylogarithms $\li_{\{1\}_k}^{ {\bf a}_k}(1/z)$, for $k\geq 0$, are linearly independent 
over the field $\mzero$ of functions meromorphic at 1.
\end{lem}

\begin{proof}
To begin with, let us consider for any $p\geq 0$ the set $\calF_p$ of all functions 
analytic on $\C\setminus [0,1]$ that can be written as $\sum_{i=0}^p h_i(z) (\log(1-\frac1z))^i$ 
where $h_0(z)$, \ldots, $h_p(z)$ are functions holomorphic  on $\C\setminus [0,1]$ and at $z=1$. 
Of course all functions  holomorphic on $\C\setminus [0,1]$ and at $z=1$ belong to $\calF_0$, 
and $\li_1(1/z) = - \log(1-\frac1z)$ belongs to $\calF_1$. We claim that for any $p\geq 0$, 
for any $\alpha_1,\ldots,\alpha_{p-1} \in\{\ell,s\}$ and any $b_1,\ldots,b_p\geq 1$, we have
$$\li_{b_1b_2\cdots b_p}^{a_1 \cdots a_{p-1}}(1/z) \in\calF_p.$$
Let us prove this claim by induction on the weight $b_1+\cdots+b_p$. We have already noticed 
that it holds if $b_1+\cdots+b_p \leq 1$. Now remark that if $f$ is   analytic on $\C\setminus [0,1]$ 
and $g\in \calF_{p}$ are such that $f'(z) = \frac{-1}z g(z)$ then $f\in\calF_p$, because $\calF_p$ 
is stable under primitivation and products with functions holomorphic at 1. On the other hand, if 
 $f'(z) = \frac{1}{1-z} g(z)$ or $f'(z) = \frac{1}{z(1-z)} g(z)$ then $f\in\calF_{p+1}$. 
Using the differentiation rules for polylogarithms stated at the beginning of \S~ \ref{sec:diff}, this proves the claim.

\bigskip

Now assume that for some $k\geq 1$ the function $\li_{\{1\}_k}^{ {\bf a}_k}(1/z)$ is 
 a linear combination over $\mzero$ of the $\li_{\{1\}_j}^{ {\bf a}_j}(1/z)$ for $0\leq j \leq k-1$. 
Using the claim this implies $\li_{\{1\}_k}^{ {\bf a}_k}(1/z) \in \calF_{k-1}$. Now applying 
Eq. \eqref{eqdeplin} as many times as needed one can write 
$\li_{\{1\}_k}^{ {\bf a}_k}(1/z)-\li_{\{1\}_k}^{ \{s\}_{k-1}}(1/z)$ as a $\Z$-linear combination 
of extended multiple polylogarithms of depth $k-1$; applying 
the claim again proves that 
 $ \li_{\{1\}_k}^{ \{s\}_{k-1}}(1/z) = (-1)^{k} \big( \log(1-\frac1z)\big)^k$ 
belongs to $\calF_{k-1}$ (this identity belongs to the folklore and is readily proved 
by induction and differentiation). 
But this provides 
a non-trivial linear relation, with coefficients holomorphic at 1, between powers 
of the function $\log(1-\frac1z)$. This is impossible since $\log(z)$ is transcendental 
over the field of functions meromorphic at the origin. This contradiction 
concludes the proof of Lemma~\ref{lemmerom}. 
\end{proof}

\section{Weight functions of multiple polylogarithms}\label{sec:weight}

In this section we study the weight functions of multiple polylogarithms and compute some of them. 
This part is at the heart of the proof of Theorem \ref{theo:1}, since {\em weights obey the same 
derivation  rules as the corresponding polylogarithms} (see below).  

If ${\bf b}=\emptyset$, 
$\li_{\emptyset}^{\boldsymbol{a}}(z)=1/(1-z)$ and none of the considerations below apply. From now on, 
we consider non-empty words ${\bf b}$. 
It is well-known that  usual 
multiple polylogarithms $\li_{{\bf b}}^{\boldsymbol{a}}(z)$ (with $\boldsymbol{a}=ss\cdots s$) 
can be analytically continued to the cut plane $\mathbb C \setminus [1,+\infty)$. They vanish at $z=0$ and 
their growth as $z\to \infty$ is 
at most a power of $\log(z)$, with $0< \arg(z)< 2\pi$.  
Moreover, the function defined on the cut by
$$
\lim_{y\to 0+}
\left[\li_{{\bf b}}^{ss\cdots s}\left(x+iy\right)- 
\li_{{\bf b}}^{ss\cdots s}\left(x-iy\right)\right]
$$
is $C^{\infty}$ on $(1,+\infty)$ with at most a (power of) logarithm singularity at $x=1$ and $x=\infty$. All these 
properties also hold for $\li_{{\bf b}}^{\boldsymbol{a}}(z)$ for any word ${\bf a}$ because such functions are simply 
linear combinations with rational coefficients of the $\li_{{\bf b}}^{ss\cdots s}(z)$ (using repeatedly Eq. \eqref{eqdeplin} above).

As an (important) application, we prove the following lemma.

\begin{lem}\label{lem:weight} For any fixed $z\in \mathbb C\setminus [0,1]$, 
any ${\bf a}$ and any ${\bf b}\neq \emptyset$, we have
\begin{equation}\label{eq:intpoids}
\li_{{\bf b}}^{\boldsymbol{a}}\left(\frac1z\right) = \int_0^1 
\frac{\om_{{\bf b}}^{\boldsymbol{a}}(x)}{z-x} \dd x,
\end{equation}
where 
\begin{equation}\label{eq:expressionpoids}
\om_{{\bf b}}^{\boldsymbol{a}}(x) :=\frac1{2i \pi}\lim_{y\to 0+}
\left[\li_{{\bf b}}^{\boldsymbol{a}}\left(\frac{1}{x}+iy\right)- 
\li_{{\bf b}}^{\boldsymbol{a}}\left(\frac{1}{x}-iy\right)\right] \in L^1([0,1]).
\end{equation}
The {\em weight function}  $\om_{{\bf b}}^{\boldsymbol{a}}(x)$ is $C^\infty$ on $(0,1)$, with at most (power of) logarithm 
singularities at $x=0$ and $x=1$.
\end{lem}
\begin{proof}
For any fixed $z\in \mathbb C\setminus [1,+\infty)$, let us consider
the Cauchy representation formula 
$$
\li_{{\bf b}}^{\boldsymbol{a}}(z) = \frac z{2i\pi}\int_{\mathcal{C}} 
\frac{\li_{{\bf b}}^{\boldsymbol{a}}(t)}{t(t-z)} \dd t,
$$
where $\mathcal{C}$ is any simple closed curve surrounding $z$ and not crossing the cut $[1,+\infty)$. We can deform 
$\mathcal{C}$ to a simple closed curve defined as follows: 
given $\varepsilon>0$ and $R>0$ (such that $\vert z\vert <R$), we glue together
two straightlines $[1+i\varepsilon,R+i\varepsilon]$, 
$[1-i\varepsilon+R,R-i\varepsilon]$, a semi-circle of center $1$ and diameter $[1-i\varepsilon, 1+i\varepsilon]$ and an arc 
of circle of center $0$ passing through $R+i\varepsilon$ and $R-i\varepsilon$ 
(both arcs not crossing $[1,+\infty)$). The analytic 
properties of $\li_{{\bf b}}^{\boldsymbol{a}}(z)$ are such that we can let $\varepsilon\to 0$ and $R\to \infty$ 
to get the representation
\begin{align*}
\li_{{\bf b}}^{\boldsymbol{a}}(z) 
& = z\int_1^{\infty} \frac{\om_{{\bf b}}^{\boldsymbol{a}}(1/t)}{t(t-z)} \dd t\\
& = z\int_0^{1} \frac{\om_{{\bf b}}^{\boldsymbol{a}}(x)}{1-zx} \dd x\qquad (\mbox{by letting }x=1/t),
\end{align*}
where $\om_{{\bf b}}^{\boldsymbol{a}}(x)$ is defined by~\eqref{eq:expressionpoids}. 
We obtain~\eqref{eq:intpoids} by changing $z$  to $1/z$.
\end{proof}
(This proof is not specific to multiple polylogarithms. Such weighted integral representations 
are known as Stieltjes representations; see~\cite[p. 591, Theorem 12.10d]{Henrici}.)

\medskip

We note two important consequences of the expression~\eqref{eq:expressionpoids} 
for $\om_{{\bf b}}^{\boldsymbol{a}}(x)$. To begin with,  if 
$$
\frac{\dd }{\dd z}\left[\li_{{\bf b}}^{\boldsymbol{a}}\left(\frac1z\right)\right] 
= R(z)\li_{{\bf b}'}^{\boldsymbol{a}'}\left(\frac1z\right), 
$$
then
$$
\frac{\dd }{\dd x}\om_{{\bf b}}^{\boldsymbol{a}}(x) 
= R(x)\om_{{\bf b}'}^{\boldsymbol{a}'}(x)
$$
where the function $R(z)$ is one of $\displaystyle -\frac{1}{z},\frac{1}{1-z}$ and  
$\displaystyle \frac{1}{z(1-z)}$ (see \S~ \ref{sec:diff}). In other words, {\em weights obey the same 
derivation  rules as the corresponding polylogarithms.} This observation will be crucial 
in \S~\ref{subsecfirstred}. Moreover, we also remark that  if 
the value $\li_{{\bf b}}^{\boldsymbol{a}}(1)$ is finite, then 
$\om_{{\bf b}}^{\boldsymbol{a}}(1)=0$.

\begin{lem}\label{lem:1} For any $x\in (0,1)$ and any integer $k\ge 0$, we have
\begin{equation}\label{eq:1}
\om_{\{1\}_{2k}}^{\{\ell s\}_{k-1}\ell}(x)=\li_{\{1\}_{2k-1}}^{\{s\ell\}_{k-1}}(x), 
\end{equation}
\begin{equation}\label{eq:2}
\om_{\{1\}_{2k+1}}^{\{s \ell \}_{k}}(x)=\li_{\{1\}_{2k}}^{\{\ell s\}_{k-1}\ell}(x),
\end{equation}
and 
\begin{align}
\om_{2\{1\}_{2k+1}}^{\{\ell s\}_{k}\ell}(x)&=\sum_{j=0}^k \li_{1\{2\}_j}^{\{\ell\}_j}(1-x) \li_{\{1\}_{2k-2j+1}}^{\{s \ell\}_{k-j}}(x) \notag
\\
& \qquad \qquad \qquad +\sum_{j=1}^{k+1} \li_{\{2\}_j}^{\{\ell\}_{j-1}}(1-x) \li_{\{1\}_{2k-2j+2}}^{\{\ell s\}_{k-j}\ell}(x) \label{eq:3}
\\ 
&= -\li_{2\{1\}_{2k}}^{\{s \ell \}_{k}}(x) + \li_{\{2\}_{k+1}}^{\{\ell\}_k}(1).\label{eq:3alternative}
\end{align}

\end{lem}
\begin{proof} Equations~\eqref{eq:1} and~\eqref{eq:2} are readily checked by expanding 
$\frac{1}{z-x}=\sum_{n=0}^{\infty} \frac{x^n}{z^{n+1}}$ in the integral~\eqref{eq:intpoids}.  
To prove~\eqref{eq:3}, we remark that both sides differentiate to the same function 
$
-\frac1x \om_{\{1\}_{2k+2}}^{\{\ell s\}_k\ell}(x)=-\frac1x \li_{\{1\}_{2k+1}}^{\{s \ell \}_{k}}(x)$, since all
  functions but this precise one are 
killed by telescoping when differentiating the right 
hand side of~\eqref{eq:3}. It follows that the functions on both sides 
of~\eqref{eq:3} differ only by a constant. This constant 
must be $0$ because both sides vanish at $x=1$ (see the remark just before Lemma~\ref{lem:1}).
The same argument yields also 
$$
\om_{2\{1\}_{2k+1}}^{\{\ell s\}_{k}\ell}(x) = - \int\frac1x \li_{\{1\}_{2k+1}}^{\{s \ell \}_{k}}(x) \dd x
=-\li_{2\{1\}_{2k}}^{\{s \ell \}_{k}}(x) + C_k
$$
for some constant $C_k$. This constant is seen to be equal to 
$\li_{\{2\}_{k+1}}^{\{\ell\}_k}(1)$ by taking  $x=0$ in~\eqref{eq:3}. This proves \eqref{eq:3alternative}, 
and concludes the proof of Lemma \ref{lem:1}.
\end{proof}

In the setting of the Pad\'e problem $\mathcal{P}_{r,n}$, we define the function
$$
P_{r,n}(z) = \sum_{\rho=0}^r \bigg[ A_{\rho,r,n}(z) \om_{2\{1\}_{2\rho+1}}^{\{\ell s\}_\rho\ell}(z) 
+ B_{\rho,r,n}(z) \om_{\{1\}_{2\rho+2}}^{\{\ell s\}_\rho\ell}(z)+C_{\rho,r,n}(z)
\om_{\{1\}_{2\rho+1}}^{\{s\ell\}_\rho}(z)\bigg]
$$
obtained from $S_{r,n}$ by replacing every polylogarithm with its weight (see Lemma \ref{lem:3} below).
By~\eqref{eq:1}, \eqref{eq:2} and~\eqref{eq:3alternative}, this function $P_{r,n}$
is analytic on the disk $\vert z\vert<1$, with a (power of) logarithm singularity at $z=1$. In particular, 
it is in $L^1([0,1])$. The following lemma is an immediate consequence 
of~\eqref{eq:1}, \eqref{eq:2},~\eqref{eq:3} and   the definition of $U_{j,r,n}(z)$ and $V_{j,r,n}(z)$. 
As in the rest of the paper, we continue analytically all polylogarithms to $\C\setminus [1,+\infty)$.

\begin{lem}\label{lem:2} 
For any $z\in \mathbb C \setminus  [1,+\infty)$, 
$$
P_{r,n}(z)=\sum_{j=0}^r \bigg[U_{j,r,n}(z)\li_{\{1\}_{2j+1}}^{\{s\ell\}_j}(z) 
+ V_{j,r,n}(z)\li_{\{1\}_{2j}}^{\{\ell s\}_{j-1}\ell}(z)\bigg].
$$
\end{lem}

We conclude this section with the precise connection between $ P_{r,n}(z)$ and $S_{r,n}(z)$.
\begin{lem} \label{lem:3} 
In the setting of the Pad\'e problem $\mathcal{P}_{r,n}$, for any $z\in \mathbb C \setminus [0,1]$ 
we have
$$
S_{r,n}(z)=\int_0^1 \frac{P_{r,n}(x)}{z-x} \dd x.
$$
\end{lem}
\begin{proof} By definition of $S_{r,n}(z)$ and Lemma \ref{lem:weight}, for any $z\in \mathbb C \setminus [0,1]$  we have
\begin{align*}
S_{r,n}(z)&=\sum_{\rho=0}^r \bigg[ A_{\rho,r,n}(z) \int_0^1 \frac{\om_{2\{1\}_{2\rho+1}}^{\{\ell s\}_\rho\ell}(x)}{z-x} \dd x
+ B_{\rho,r,n}(z) \int_0^1\frac{\om_{\{1\}_{2\rho+2}}^{\{\ell s\}_\rho\ell}(x)}{z-x} \dd x
\\ & \hspace{4cm}+C_{\rho,r,n}(z)
\int_0^1\frac{\om_{\{1\}_{2\rho+1}}^{\{s\ell\}_\rho}(x)}{z-x} \dd x\bigg] + D_{r,n}(z)
\\ &= \int_0^1 \frac{P_{r,n}(x)}{z-x} \dd x + \sum_{\rho=0}^r \int_0^1 
\bigg[\frac{A_{\rho,r,n}(z)-A_{\rho,r,n}(x)}{z-x}\om_{2\{1\}_{2\rho+1}}^{\{\ell s\}_\rho\ell}(x) 
\\&+ 
\frac{B_{\rho,r,n}(z)-B_{\rho,r,n}(x)}{z-x}\om_{\{1\}_{2\rho+2}}^{\{\ell s\}_\rho\ell}(x) + 
\frac{C_{\rho,r,n}(z)-C_{\rho,r,n}(x)}{z-x}\om_{\{1\}_{2\rho+1}}^{\{s\ell \}_\rho}(x)
\bigg]\dd x + D_{r,n}(z).
\end{align*}
Hence, 
\begin{equation}\label{eq:Spoly}
S_{r,n}(z)=\int_0^1 \frac{P_{r,n}(x)}{z-x} \dd x + \textup{polynomial}(z) .
\end{equation}
But, as $z\to \infty$, $S_{r,n}(z)=\mathcal{O}(1/z)$ and 
$\int_0^1 \frac{P_{r,n}(x)}{z-x} \dd x \to 0$  (because 
$P_{r,n}(x)\in L^1([0,1])$, as noticed above). Therefore, the polynomial in~\eqref{eq:Spoly} is identically $0$ and 
this completes the proof of Lemma \ref{lem:3}.
\end{proof}

\section{Resolution of the Pad\'e problem $\mathcal{P}_{r,n}$}\label{sec:proofthm1}

In this section we prove Theorem \ref{theo:1}, using the tools of 
\S\S~\ref{sec:technicalities} and \ref{sec:weight}. Starting with a 
solution $S_{r,n}(z)$ of the Pad\'e problem $\Prn$, we apply the differential 
operator $\frac{z^{n+1}}{n!} \big( \frac{ \dd }{\dd z}\big)^{n+1}$ and prove in 
\S\S~ \ref{subsecfirstred} and \ref{subsecsecondred} that the resulting function 
is a solution of another Pad\'e approximation problem, denoted by $\Qrn$ and 
stated in \S~\ref{sec:sorokinpi2}. Then we observe in  \S~\ref{sec:sorokinpi2} 
that $\Qrn$ is nothing but Sorokin's problem~\cite{sorokin1}  for $\pi^2$, denoted 
by $\soro$, up to a change of variable $z\mapsto 1-z$. Since Sorokin has proved 
that $\soro$ has a unique solution up to proportionality, the same result holds for $\Qrn$ and $\Prn$.

To conclude the proof of Theorem \ref{theo:1}, we deduce in \S\S~\ref{subsec54} and \ref{subseccalcint} the integral
 representation \eqref{eq:intro6} of $S_{r,n}(z)$ from Sorokin's  integral
 representation of the solution of $\soro$, using the integral operator introduced in \S~\ref{subsec32}.

\subsection{First reduction} \label{subsecfirstred}

Let $S_{r,n}(z)$ be a solution  of the Pad\'e problem $\Prn$. 
By Lemma~\ref{lem:4}, there exist some polynomials $\check{A}_{\rho,r,n}(z)$,  $\check{B}_{\rho,r,n}(z)$ 
and  $\check{C}_{r,n}(z)$ of degree $\le 2n+1$ such that
\begin{multline}
\widehat{S}_{r,n}(z):=\frac{z^{n+1}}{n!}S^{(n+1)}_{r,n}(z) 
= \sum_{\rho=0}^r \bigg[\frac{\check{A}_{\rho,r,n}(z)}{(1-z)^{n+1}} \li_{\{1\}_{2\rho+2}}^{\{\ell s\}_\rho \ell}\bigg(\frac1z\bigg) 
\\
+\frac{\check{B}_{\rho,r,n}(z)}{(1-z)^{n+1}} \li_{\{1\}_{2\rho+1}}^{\{s\ell\}_\rho}\bigg(\frac1z\bigg)  \bigg] 
+ \frac{\check{C}_{r,n}(z)}{(1-z)^{n+1}} = \mathcal{O}\bigg(\frac{1}{z^{(r+1)(n+1)}}\bigg).\label{eq:66}
\end{multline}
As in \S~\ref{sec:weight} we consider the function $P_{r,n}(z)$ defined by 
$$
P_{r,n}(z) = \sum_{\rho=0}^r \bigg[ A_{\rho,r,n}(z) \om_{2\{1\}_{2\rho+1}}^{\{\ell s\}_\rho\ell}(z) 
+ B_{\rho,r,n}(z) \om_{\{1\}_{2\rho+2}}^{\{\ell s\}_\rho\ell}(z)+C_{\rho,r,n}(z)
\om_{\{1\}_{2\rho+1}}^{\{s\ell\}_\rho}(z)\bigg].
$$
Since it is obtained from $S_{r,n}$ by replacing each polylogarithm by its weight, 
  it  obeys the same derivation rules (see the remark before Lemma \ref{lem:1}). This implies that  
\begin{align}
\widehat{P}_{r,n}(z):=\frac{z^{n+1}}{n!}P^{(n+1)}_{r,n}(z) 
&= \sum_{\rho=0}^r \bigg[\frac{\check{A}_{\rho,r,n}(z)}{(1-z)^{n+1}} \om_{\{1\}_{2\rho+2}}^{\{\ell s\}_\rho \ell}(z) 
+\frac{\check{B}_{\rho,r,n}(z)}{(1-z)^{n+1}} \om_{\{1\}_{2\rho+1}}^{\{s\ell\}_\rho}(z)  \bigg] \notag
\\
&=\sum_{\rho=0}^r \bigg[\frac{\check{A}_{\rho,r,n}(z)}{(1-z)^{n+1}} \li_{\{1\}_{2\rho+1}}^{\{s\ell\}_\rho}(z) 
+\frac{\check{B}_{\rho,r,n}(z)}{(1-z)^{n+1}} \li_{\{1\}_{2\rho}}^{\{\ell s\}_{\rho-1}\ell}(z)  \bigg] \label{eq:5}
\end{align}
with the same polynomials $\check{A}_{\rho,r,n}(z)$ and $\check{B}_{\rho,r,n}(z)$; here we have used Eqs.~\eqref{eq:1} 
and~\eqref{eq:2} in Lemma~\ref{lem:1} to compute the weights. 

Now, by Lemmas \ref{lemannulation},~\ref{lem:2} and the Pad\'e conditions at $z=1$ in $\mathcal{P}_{r,n}$ 
for $U_{j,r,n}$ and $V_{j,r,n}$, the function $\widehat{P}_{r,n}(z)$ is necessarily of the form
\begin{equation}\label{eq:6}
\widehat{P}_{r,n}(z)=\sum_{j=0}^r \bigg[h_{2j+1}(z) \li_{\{1\}_{2j+1}}^{\{s\ell\}_j}(z)
+ h_{2j}(z) \li_{\{1\}_{2j}}^{\{\ell s\}_{j-1}\ell}(z)  \bigg]
\end{equation}
for some  functions $h_j$ holomorphic at $z=1$. Now we have obtained two expressions for 
$\widehat{P}_{r,n}(z)$, namely Eqns. \eqref{eq:5} and~\eqref{eq:6}. Using Lemma \ref{lemmerom} they have to coincide, that is  
$ \frac{\check{A}_{\rho,r,n}(z)}{(1-z)^{n+1}} = h_{2\rho+1}(z)$ and 
$\frac{\check{B}_{\rho,r,n}(z)}{(1-z)^{n+1}} =  h_{2\rho}(z)$ for any $\rho=0, \ldots, r$. Therefore $(1-z)^{n+1}$ 
divides $\check{A}_{\rho,r,n}(z)$ and $\check{B}_{\rho,r,n}(z)$.

We now claim that $(1-z)^{n+1}$ also divides $\check{C}_{r,n}(z)$. To prove  this, we use the integral 
representation for $S_{r,n}(z)$ given by Lemma~\ref{lem:3}.  Differentiating $n+1$ times under the integral, we obtain
$$
\widehat{S}_{r,n}(z)=(n+1)(-z)^{n+1}\int_0^1 \frac{P_{r,n}(x)}{(z-x)^{n+2}} \dd x.
$$
Again by Lemma~\ref{lem:2} and the Pad\'e conditions at $z=1$ in $\mathcal{P}_{r,n}$ 
for $U_{r,n,j}$ and $V_{r,n,j}$, we deduce that   
$$
P_{r,n}(x)=\mathcal{O}\big((1-x)^{n+1} (1+\vert \log(1-x)\vert^{2r+1})\big)
$$ 
as $x\to 1$, $x<1$.  Therefore 
the singularity of $\widehat{S}_{r,n}(z)$ at $z=1$  is  
at most a power of logarithm. 
The expression~\eqref{eq:66} for $\widehat{S}_{r,n}(z)$, together with 
the above deductions made for $\check{A}_{\rho,r,n}(z)$ and $\check{B}_{\rho,r,n}(z)$, implies the claim.

We can summarize the above results as follows:  there exist   polynomials $\widehat{A}_{\rho,r,n}(z)$,  
$\widehat{B}_{\rho,r,n}(z)$ ($\rho \in\{0, \ldots, r\}$) and $\widehat{C}_{r,n}(z)$, all of 
degree at most $n$, such that
\begin{multline}
\widehat{S}_{r,n}(z)
= \sum_{\rho=0}^r \bigg[\widehat{A}_{\rho,r,n}(z) \li_{\{1\}_{2\rho+2}}^{\{\ell s\}_\rho \ell}\bigg(\frac1z\bigg) 
\\
+\widehat{B}_{\rho,r,n}(z) \li_{\{1\}_{2\rho+1}}^{\{s\ell\}_\rho}\bigg(\frac1z\bigg)  \bigg] 
+ \widehat{C}_{r,n}(z) = \mathcal{O}\bigg(\frac{1}{z^{(r+1)(n+1)}}\bigg).\label{eq:7}
\end{multline}

\subsection{Second reduction} \label{subsecsecondred}

We want to find further Pad\'e conditions involving the polynomials $\widehat{A}_{\rho,r,n}(z)$,  
$\widehat{B}_{\rho,r,n}(z)$ ($\rho \in\{0, \ldots, r\}$) and $\widehat{C}_{r,n}(z)$. For this, we form the functions
$$
Q_{j,r,n}:=\sum_{\rho=j}^r \bigg[-A_{\rho,r,n}(z) \li_{2\{1\}_{2\rho-2j}}^{\{s\ell\}_{\rho-j}}(z) 
+ B_{\rho,r,n}(z) \li_{\{1\}_{2\rho-2j+1}}^{\{s\ell\}_{\rho-j}}(z) + 
C_{\rho,r,n}(z) \li_{\{1\}_{2\rho-2j}}^{\{\ell s\}_{\rho-j-1}\ell}(z)\bigg]
$$
where $j=0, \ldots, r$, and  $A_{\rho,r,n}(z)$, $B_{\rho,r,n}(z)$, $C_{\rho,r,n}(z)$ are 
the polynomials in our initial Pad\'e problem $\mathcal{P}_{r,n}$. Each 
$Q_{j,r,n}(z)$ is holomorphic at $z=0$ and the rules of differentiation of multiple 
 polylogarithms (see \S~\ref{sec:diff})  
 show that
\begin{multline*}
\widehat{Q}_{j,r,n}(z):=\frac{z^{n+1}}{n!}Q_{j,r,n}^{(n+1)}(z)
\\
=\sum_{\rho=j}^r \bigg[\widehat{A}_{\rho,r,n}(z) \li_{\{1\}_{2\rho-2j+1}}^{\{s\ell\}_{\rho-j}}(z) 
+ \widehat{B}_{\rho,r,n}(z) \li_{\{1\}_{2\rho-2j}}^{\{\ell s\}_{\rho-j-1}\ell}(z) \bigg] =  \mathcal{O}(z^{n+1})
\end{multline*}
for all $j=0, \ldots r$. The main point here is that the polynomials $\widehat{A}_{\rho,r,n}(z) $ 
and $ \widehat{B}_{\rho,r,n}(z)$ are the same as 
in Eq. \eqref{eq:7}.

\subsection{The intermediate Pad\'e problem $\mathcal{Q}_{r,n}$}\label{sec:sorokinpi2}

The previous two sections show that any solution $S_{r,n}(z)$  to the problem $\mathcal{P}_{r,n}$ 
yields (by differentiating $n+1$ times and multiplying by $z^{n+1}/n!$) a solution to the following problem: given   
non-negative integers $r$ and $n$,  find polynomials $\widehat{A}_{\rho,r,n}(z)$, 
$ \widehat{B}_{\rho,r,n}(z)$ (for $0 \leq \rho \leq r$) and $ \widehat{C}_{r,n}(z) $, 
of degrees $\leq n$, such that the following holds:
\begin{align}
\widehat{S}_{r,n}(z)&
:= \sum_{\rho=0}^r \bigg[\widehat{A}_{\rho,r,n}(z) \li_{\{1\}_{2\rho+2}}^{\{\ell s\}_\rho \ell}\bigg(\frac1z\bigg) 
 +\widehat{B}_{\rho,r,n}(z) \li_{\{1\}_{2\rho+1}}^{\{s\ell\}_\rho}\bigg(\frac1z\bigg)  \bigg]  \notag
\\
&\hspace{8cm} + \widehat{C}_{r,n}(z) = \mathcal{O}\bigg(\frac{1}{z^{(r+1)(n+1)}}\bigg),  \notag
\\
\widehat{Q}_{j,r,n}(z)&:=\sum_{\rho=j}^r \bigg[\widehat{A}_{\rho,r,n}(z) 
\li_{\{1\}_{2\rho-2j+1}}^{\{s\ell\}_{\rho-j}}(z) + \widehat{B}_{\rho,r,n}(z) 
\li_{\{1\}_{2\rho-2j}}^{\{\ell s\}_{\rho-j-1}\ell}(z) \bigg] \notag
\\
&\hspace{8cm} = \mathcal{O}(z^{n+1}), \quad j=0, \ldots, r.  \notag
\end{align}
We shall denote this Pad\'e approximation problem by  $\mathcal{Q}_{r,n}$. 
It amounts to solving a linear system of $(3r+4)(n+1)-1$ equations 
in $(3r+4)(n+1)$ unknowns (the coefficients of the polynomials). Hence it has a least one non trivial 
solution and our next task is to prove that is has exactly one solution up to a multiplicative constant.

To do so, we will identify the problem with one already solved by Sorokin~\cite{sorokin1}. 
We first observe the effect of changing $z$ to $1-z$ in the Pad\'e problem  $\mathcal{Q}_{r,n}$.

\begin{lem}\label{lem:6} For any $z\in \mathbb C\setminus [0,1]$, we have
\begin{align*}
\li_{\{1\}_{2\rho+1}}^{\{s\ell \}_\rho}\bigg(\frac1z\bigg)&
  = (-1) ^{\rho+1}\li_{1\{2\}_{\rho}}^{\{s\}_\rho}\bigg(\frac1{1-z}\bigg), 
\\
\li_{\{1\}_{2\rho+2}}^{\{\ell s\}_\rho \ell}\bigg(\frac1z\bigg)&
  = (-1) ^{\rho+1}\li_{\{2\}_{\rho+1}}^{\{s\}_\rho}\bigg(\frac1{1-z}\bigg).
\end{align*}
\end{lem}
\begin{proof} We prove these identities by induction on $\rho$. They 
hold trivially  for $\rho=0$ and by differentiation of both sides at level $\rho$, we get 
the identity at level $\rho-1$. We deduce that the identity at level $\rho$ 
holds, up to some additive constant. This constant must be $0$ because both sides vanish at $z=\infty$.
\end{proof}

Therefore, when we change $z$ to $1-z$, the Pad\'e problem  $\mathcal{Q}_{r,n}$  becomes
\begin{align*}
\widehat{S}_{r,n}(1-z)&
:= \sum_{\rho=0}^r(-1)^{\rho+1} \bigg[\widehat{A}_{\rho,r,n}(1-z) \li_{\{2\}_{\rho+1}}^{\{s\}_\rho}\bigg(\frac1z\bigg) 
 +\widehat{B}_{\rho,r,n}(1-z) \li_{1\{2\}_{\rho}}^{\{s\}_\rho}\bigg(\frac1z\bigg)  \bigg] 
\\
&\qquad\qquad
+ \widehat{C}_{r,n}(1-z) = \mathcal{O}\bigg(\frac{1}{(1-z)^{(r+1)(n+1)}}\bigg)
=\mathcal{O}\bigg(\frac{1}{z^{(r+1)(n+1)}}\bigg)
\\
\widehat{Q}_{j,r,n}(1-z)&:=\sum_{\rho=j}^r \bigg[\widehat{A}_{\rho,r,n}(1-z) \li_{\{1\}_{2\rho-2j+1}}^{\{s\ell\}_{\rho-j}}(1-z)
\\
& \qquad\qquad + \widehat{B}_{\rho,r,n}(1-z) 
\li_{\{1\}_{2\rho-2j}}^{\{\ell s\}_{\rho-j-1}\ell}(1-z) \bigg] =  \mathcal{O}((1-z)^{n+1}), \quad j=0, \ldots, r.
\end{align*}
Let us define
\begin{align*}
\widetilde{A}_{\rho,r,n}(z)&=(-1)^{\rho+1}\widehat{A}_{\rho,r,n}(1-z), \quad
\widetilde{B}_{\rho,r,n}(z)=(-1)^{\rho+1}\widehat{B}_{\rho,r,n}(1-z),
\\
\widetilde{C}_{r,n}(z)&=\widehat{C}_{r,n}(1-z), \quad
\widetilde{S}_{r,n}(z)=\widehat{S}_{r,n}(1-z), \quad
\widetilde{Q}_{j,r,n}(z)=-\widehat{Q}_{j,r,n}(1-z).
\end{align*}
With these notations, the Pad\'e problem  $\mathcal{Q}_{r,n}$ now reads
\begin{align*}
\widetilde{S}_{r,n}(z):=&\sum_{\rho=0}^r \bigg[\widetilde{A}_{\rho,r,n}(z)
\li_{\{2\}_{\rho+1}}^{\{s\}_\rho}\bigg(\frac1z\bigg) + \widetilde{B}_{\rho,r,n}(z)
\li_{1\{2\}_{\rho}}^{\{s\}_\rho}\bigg(\frac1z\bigg)\bigg] + \widetilde{C}_{r,n}(z) = \mathcal{O}\bigg(\frac{1}{z^{(r+1)(n+1)}}\bigg)
\\
\widetilde{Q}_{j,r,n}(z)&:=\sum_{\rho=j}^r (-1)^\rho\bigg[\widetilde{A}_{\rho,r,n}(z) \li_{\{1\}_{2\rho-2j+1}}^{\{s\ell\}_{\rho-j}}(1-z)
+ \widetilde{B}_{\rho,r,n}(z) \li_{\{1\}_{2\rho-2j}}^{\{\ell s\}_{\rho-j-1}\ell}(1-z) \bigg]
\\ &\hspace{8cm}=  \mathcal{O}((1-z)^{n+1}), \quad j=0, \ldots, r.
\end{align*}
In spite of different notations, we recognize here Sorokin's problem~\cite{sorokin1} 
for $\pi^2$ of weight $2r+2$,  which we denote by $\soro$ from now on. 
Sorokin proved that this problem has a unique solution up to proportionality. Therefore the same 
property holds for $\Qrn$, and also for $\Prn$. This concludes the proof of Theorem \ref{theo:1}, 
except for the integral representation \eqref{eq:intro6}  of $S_{r,n}(z)$ that we shall prove now.

\subsection{Hypergeometric integrals for $\tilde{S}_{r,n}(z)$ and $S_{r,n}(z)$} \label{subsec54}

Sorokin has   found an explicit integral formula for the solution 
$\widetilde{S}_{r,n}(z)$ of his Pad\'e problem $\soro$ stated in \S~\ref{sec:sorokinpi2} 
(see~\cite[Lemma 17, p. 1835]{sorokin1}), namely
\begin{equation}\label{eq:intro2}
\widetilde{S}_{r,n}(z)=(-1)^{(r+1)n}\int_{[0,1]^{2r+2}} \prod_{j=1}^{r+1} \frac{x_j^n(1-x_j)^n y_j^n(1-y_j)^n}
{\big(\frac{z}{x_1y_1\cdots x_{j-1}y_{j-1}}-x_jy_j\big)^{n+1}} 
\dd x_j\dd y_j.
\end{equation}
In this and the next sections we shall deduce from it the  integral expression \eqref{eq:intro6}  
of $S_{r,n}(z)$, using the relation
\begin{equation}\label{eq:11}
\frac{z^{n+1}}{n!}S_{r,n}^{(n+1)}(z)= \widetilde{S}_{r,n}(1-z)
\end{equation}
and the integral operator defined in \S~\ref{subsec32}. 

\medskip

To begin with, we recall that 
Sorokin solved his Pad\'e approximation problem $\soro$ recursively 
and showed that, for any integer  $r\ge 1$ and 
any $z\in \mathbb C\setminus [0,1]$, 
\begin{equation}\label{eq:13}
\tilde{S}_{r-1,n}(z) = \frac{1}{n!^2}z^{n+1}(1-z)^{n+1}\big(z^{n+1}\tilde{S}_{r,n}^{(n+1)}(z)\big)^{(n+1)}
\end{equation}
and 
$$
\tilde{S}_{0,n}(z)=\int_0^1\int_0^1 \frac{x^n(1-x)^ny^n(1-y)^n}{(z-xy)^{n+1}} \dd x\dd y.
$$
It is not hard to see that, with the notation of \S~\ref{subsec32}, we have for $z\in\C\setminus[0,1]$:
\begin{equation}\label{eq59bis}
\tilde{S}_{0,n}(z) = \BBB^{n+1}_{n+1,0}  \left(\int_0^1 \frac{x^n(1-x)^n}{(z-x)^{n+1}}\right)=
\BBB^{n+1}_{n+1,0}  \BBB^{n+1}_{n+1,n+1}  (\boldsymbol{1}),
\end{equation}
where $\boldsymbol{1}$ is the constant function equal to 1 on $\C\setminus[0,1]$. 
We can apply the general properties of  hypergeometric integrals proved in  \S~\ref{subsec32} to~\eqref{eq:13} and we get 
the following result, which 
is nothing but~\eqref{eq:intro2} written in a different language (see \S~\ref{subseccalcint} 
for details). We recall that $f^{\partial}(z):=f(1-z)$ and we denote by $\BBB^k 
= \BBB\circ\BBB \circ\cdots\circ\BBB$ the composition of an integral 
operator $\BBB$ with itself $k$ times.

\begin{prop}\label{prop:1} For any $z\in \mathbb{C}\setminus[0,1]$ and any integer $r\ge 0$, we have
\begin{equation}\label{eq:14}
\tilde{S}_{r,n}(z)= (\opa)^{r+1} (\boldsymbol{1})(z) 
\end{equation}
 and
\begin{equation}\label{eq:15}
\tilde{S}_{r,n}^{\partial}(z)=(\opb)^{r+1} (\boldsymbol{1})(z).
\end{equation}
\end{prop} 
Eq. \eqref{eq:14} follows immediately from Eq. \eqref{eq59bis} and the relation 
$$\tilde{S}_{r,n} = \opa (\tilde{S}_{r-1,n}),$$
 which is just a translation of Eq. \eqref{eq:13} (using Lemma \ref{lem:7}). Then
Eq.~\eqref{eq:15} follows from~\eqref{eq:14} by means of Lemma~\ref{lem:8}.  Now Eq.~\eqref{eq:11} reads
\begin{equation}\label{eq:15bis}
\frac{z^{n+1}}{n!}S_{r,n}^{(n+1)}(z) = \tilde{S}_{r,n}^{\partial}(z)
\end{equation}
and  $\lim_{z\to \infty} S_{r,n} (z)=0$ for any $r\ge 0$, so that Lemma \ref{lem:7} yields
$$
S_{r,n}(z)=\BBB^{ n+1}_{n+1,0} (\widetilde{S}_{r,n}^{\partial})(z).
$$
Hence, by~\eqref{eq:15} in Proposition~\ref{prop:1}, we obtain the following result 
(using also Lemma \ref{lemenun} to take limits as $z\to 1$).
\begin{prop}\label{prop:2} For any $z\in \mathbb{C}\setminus[0,1]$ and any integer $r\ge 0$, we have 
\begin{equation}\label{eq:16}
S_{r,n}(z)=\BBB^{ n+1}_{n+1,0} (\opb)^{r+1} (\boldsymbol{1})(z).
\end{equation}
 Moreover, both sides of~\eqref{eq:16} are defined 
and equal for $z=1$.
\end{prop}

\subsection{Explicit multiple integrals} \label{subseccalcint}

The integral expression for $S_{r,n}(z)$ given in Theorem~\ref{theo:1} is simply the  
explicit ``expansion'' of the formula \eqref{eq:16} given in Proposition~\ref{prop:2} above. Let us provide 
details on this expansion.

For any  function $F$ analytic on $\C\setminus [0,1]$ and at infinity,  Eq. \eqref{eq:12} in \S~\ref{subsec32} reads
$$
\BBB^{n+1}_{n+1,n+1} (F)(z)= (-1)^{n+1}  \int_0^1 \frac{u^n(1-u)^n}{(u-z)^{n+1}}F\Big(\frac zu\Big) \dd u.
$$
This function $\BBB^{n+1}_{n+1,n+1} (F)(z)$ is analytic on $\C\setminus [0,1]$ and at infinity, 
and vanishes to an order $\geq n+1$ at $\infty$ (using Lemma \ref{lem:7}). The same property 
can be proved  in the same way for the following function:
\begin{align*}
 \opb (F)(z)&=z^{n+1}\int_0^1 \frac{v^{-1}(1-v)^n}{(v-z)^{n+1}}
 \int_0^1\frac{u^n(1-u)^n}{(u-z/v)^{n+1}}F\Big(\frac z{uv}\Big) \dd u \dd v 
\\
&=z^{n+1}\int_0^1\int_0^1\frac{v^n(1-v)^nu^n(1-u)^n}{(v-z)^{n+1}(uv-z)^{n+1}}F\Big(\frac z{uv}\Big) \dd u \dd v.
\end{align*}
By induction on $r\ge 0$ this implies, using Eq. \eqref{eq:15}:
\begin{multline*}
\tilde{S}_{r,n}^{\partial}(z) = (\opb)^{r+1} (\boldsymbol{1})(z) =  z^{(r+1)(n+1)}
\\\times \int\limits_{[0,1]^{2(r+1)}}\frac{\displaystyle \prod_{j=1}^{r+1}\big((u_jv_j)^{(r-j+2)(n+1)-1}(1-u_j)^n(1-v_j)^n\big)}
{\displaystyle \prod_{j=1}^{r+1} \big((z-u_1v_1\cdots u_{j-1}v_{j-1}u_j)^{n+1}(z-u_1v_1\cdots u_{j}v_{j})^{n+1}\big)}
  \dd {\bf u} \dd {\bf v}.
\end{multline*}
Therefore the equality
$$
\BBB^{ n+1}_{n+1,0} (\opb)^{r+1} (\boldsymbol{1})(z)  = 
(-1)^{ n+1}  \int_0^1 u_0^{-1}(1-u_0)^{  n} \tilde{S}_{r,n}^{\partial}(z/u_0) \dd u_0
$$
yields, using Proposition \ref{prop:2}:
\begin{multline*}
S_{r,n}(z)=(-1)^{ n+1}z^{(r+1)  (n+1)} \\
\times \int_{[0,1]^{2r+3}} 
\frac{\displaystyle u_0^{(r+1) (n+1)-1}(1-u_0)^{ n} 
\prod_{j=1}^{r+1}\big((u_jv_j)^{(r-j+2)(n+1)-1}(1-u_j)^n(1-v_j)^n\big)}
{\displaystyle \prod_{j=1}^{r+1}\big((z-u_0u_1v_1\cdots u_{j-1}v_{j-1}u_j)^{n+1}
(z-u_0u_1v_1\cdots u_{j}v_{j})^{n+1}\big)} \dd{\bf u}\dd{\bf v}.
\end{multline*}
This completes the proof of Theorem~\ref{theo:1}.

\section{Beyond Vasilyev's conjecture: irrationality of odd zeta values} \label{subsecinf}

A natural problem is  to find a  
proof that the numbers  $\zeta(2r+1)$, $r\geq 0$, span an infinite-dimensional $\Q$-vector space~\cite{BR, rivoal}  
that would be analogous to Sorokin's proof that $\pi$ is transcendental \cite{sorokin1} 
(since Sorokin's result is equivalent to the fact that the numbers  $\zeta(2r)$, $r\geq 0$, 
span an infinite-dimensional $\Q$-vector space). In particular, such a proof would involve a 
Pad\'e approximation problem with multiple polylogarithms.

Let $\sigma$ be an integer such that $1 \leq \sigma \leq r+2$. To achieve this goal, it is 
enough to relate the very-well-poised hypergeometric series 
\begin{equation} \label{eqvwp}
\sum_{k=1}^{\infty} (k+\frac{n}2) \frac{(k-\sigma n)_{\sigma n} (k+n+1) _{\sigma n}}{(k)_{n+1}^{2r+4}},
\end{equation}
which can be used to prove the above mentioned result (see for instance \cite{fischler3}), 
to such a Pad\'e approximation problem. An analogous work has been done in \cite{firi}, where this 
series is related to a Pad\'e approximation problem involving only classical polylogarithms, namely of depth 1. 

We shall prove now that for $\sigma =1$ the  hypergeometric series  \eqref{eqvwp} is equal (up to a sign) 
to $S_{r,n}(1)$, thereby providing in this case the relation we are looking for. For any $\sigma$ we shall 
prove that this series is the value at $z=1$ of a function $S_{r,n,\sigma}(z)$ which generalizes  
$S_{r,n }(z)$; what is missing  is a Pad\'e approximation problem 
of which $S_{r,n,\sigma}(z)$  would be  a solution. We believe that a suitable generalisation of the problem  
$\Prn$ solved in Theorem \ref{theo:1} could have this property.

\bigskip

With this aim in view, we  consider the function $S_{r,n,\sigma}(z)$ defined by
$$
\frac{z^{n+1}}{n!}S_{r,n,\sigma}^{(\sigma n+1)}(z) = \tilde{S}_{r,n}^{\partial}(z)
$$
and  $\lim_{z\to \infty} S_{r,n,\sigma} (z)=0$; in this way we have $S_{r,n,1}(z)  =  S_{r,n }(z)$ 
(see Eq. \eqref{eq:15bis}).   We have
$$
S_{r,n,\sigma}(z)=\BBB^{\sigma n+1}_{n+1,0} (\widetilde{S}_{r,n}^{\partial})(z).
$$
The equality
$$
\BBB^{\sigma n+1}_{n+1,0} (\opb)^{r+1} (\boldsymbol{1})(z)  = 
(-1)^{\sigma n+1}z^{(\sigma - 1)  n } \int_0^1 u_0^{(1-\sigma)n-1}(1-u_0)^{\sigma n} \tilde{S}_{r,n}^{\partial}(z/u_0) \dd u_0
$$
yields, using Proposition \ref{prop:2}:
\begin{multline*}
S_{r,n,\sigma}(z)=(-1)^{\sigma n+1}z^{(r+\sigma)  n+r+1} \\
\times \int_{[0,1]^{2r+3}} 
\frac{\displaystyle u_0^{(r-\sigma+2) n+r}(1-u_0)^{\sigma n} 
\prod_{j=1}^{r+1}\big((u_jv_j)^{(r-j+2)(n+1)-1}(1-u_j)^n(1-v_j)^n\big)}
{\displaystyle \prod_{j=1}^{r+1}\big((z-u_0u_1v_1\cdots u_{j-1}v_{j-1}u_j)^{n+1}
(z-u_0u_1v_1\cdots u_{j}v_{j})^{n+1}\big)} \dd{\bf u}\dd{\bf v}.
\end{multline*}
This function has the following value at $z=1$:
\begin{multline*}
S_{r,n,\sigma}(1) = 
\\ (-1)^{\sigma n+1}  \int\limits_{[0,1]^{2r+3}} 
\frac{\displaystyle u_0^{(r-\sigma+2) n+r}(1-u_0)^{\sigma n} \prod_{j=1}^{r+1}\big((u_jv_j)^{(r-j+2)(n+1)-1}(1-u_j)^n(1-v_j)^n\big)}
{\displaystyle \prod_{j=1}^{r+1}
\big((1-u_0u_1v_1\cdots u_{j-1}v_{j-1}u_j)^{n+1}(1-u_0u_1v_1\cdots u_{j}v_{j})^{n+1}\big)} \dd{\bf u}\dd{\bf v}.
\end{multline*}

Using Proposition 17 of \cite{fischler2}  (which amounts to a change of variables) one obtains
$$S_{r,n,\sigma}(1) = (-1)^{\sigma n+1}  \int\limits_{[0,1]^{a-1}} \frac{\prod_{j=1}^{a-1} x_j^{\sigma n} 
(1-x_j)^n }{(1-x_1x_2\cdots x_{a-1})^{\sigma n+1} \prod_{2 \leq j \leq a-2 \atop j {\tiny \mbox{ even}}} 
(1-x_1x_2\cdots x_j)^{n+1}} \dd{\bf x}$$
with $a=2r+4$.
Then using Zlobin's result \cite{zlobin} or another change of variables 
(namely Th\'eor\`eme 10 of \cite{fischler2}), one obtains the Vasilyev-type integral 
$$S_{r,n,\sigma}(1) = (-1)^{\sigma n+1}  \int\limits_{[0,1]^{a-1}} \frac{\prod_{j=1}^{a-1} x_j^{\sigma n} (1-x_j)^n
 }{ Q_{a-1}(x_1,\cdots,x_{a-1})^{\sigma n+1} }  \dd{\bf x}.$$
Now Theorem 5 of \cite{zudilin} yields
$$S_{r,n,\sigma}(1) = (-1)^{\sigma n+1} \sum_{k=1}^{\infty} (k+\frac{n}2) \frac{(k-\sigma n)_{\sigma n} (k+n+1) _{\sigma n}}{(k)_{n+1}^a}.
$$
Up to a sign, this is exactly the very-well poised hypergeometric series \eqref{eqvwp}.

\def\refname{Bibliography}

S. Fischler, \'Equipe d'Arithm\'etique et de G\'eom\'etrie Alg\'ebrique, 
Universit\'e Paris-Sud, B\^atiment 425,
91405 Orsay Cedex, France

T. Rivoal,  Institut Fourier,  CNRS et Universit\'e Grenoble 1, 
100 rue des maths, BP 74, 38402 St Martin d'H\`eres Cedex, France 

\end{document}